\documentclass[reqno, a4]{amsart}
\usepackage{amsmath,amssymb,amsthm}
\usepackage{graphicx}

\def\lexl{<_{\mathrm{lex}}}

\def\N{\mathbb N}
\def\Z{\mathbb Z}
\def\Q{\mathbb Q}
\def\R{\mathbb R}

\def\A{\mathcal{A}}

\def\eps{\varepsilon}
\def\RR{\mathcal{R}}
\def\LL{\mathcal{L}}

\newtheorem{theorem}{Theorem}
\newtheorem{thm}[theorem]{Theorem}
\newtheorem{prop}[theorem]{Proposition}
\newtheorem{lem}[theorem]{Lemma}

\newtheorem{conj}[theorem]{Conjecture}
\newtheorem{rem}[theorem]{Remark}

\allowdisplaybreaks[1]

\begin{document}
\title{Periodic expansion of one by Salem numbers}
\date{}
\author{Shigeki Akiyama and Hachem Hichri}
\address{Shigeki Akiyama, Institute of Mathematics, University of Tsukuba, 
1-1-1 Tennodai, Tsukuba, Ibaraki, 305-8571 Japan}
\email{akiyama@math.tsukuba.ac.jp}
\address{
Hachem Hichri, Departement de Mathematiques (UR18ES15), Faculte des sciences de Monastir, Universite de Monastir, Monastir 5019, Tunisie}
\email{hichemhichri@yahoo.fr}
\maketitle

\begin{abstract}
We show that for a Salem number $\beta$ of degree $d$, 
there exists a positive constant $c(d)$
that $\beta^m$ is a Parry number for integers $m$ of natural density $\ge c(d)$.
Further, we show $c(6)>1/2$ and discuss a 
relation to the discretized rotation in dimension $4$.
\end{abstract}

Let $\beta>1$. R\'enyi \cite{Renyi1957} introduced the beta transformation on $[0,1)$ by
$$
T_{\beta}(x)=\beta x-\lfloor \beta x \rfloor.
$$
This map has long been applied in many branches of mathematics, such as number theory, dynamical system, coding theory, and computer sciences. 
The dynamical system $([0,1),T_{\beta})$ 
admits the ``Parry measure" $\mu_{\beta}$:
a unique invariant measure equivalent to the Lebesgue measure \cite{Parry1960}.
The system is ergodic with respect to $\mu_{\beta}$ and 
gives an important class of systems with explicit invariant density.  
The transformation
$T_{\beta}$ gives a representation of real numbers in a non-integer base (c.f. \cite{Frougny1970}).
Defining $x_n=\lfloor \beta T_{\beta}^{n-1}(x)\rfloor$, 
we obtain the ``greedy" expansion:
$$
x= \frac {x_1}{\beta}+\frac {x_2}{\beta^2}+\dots
$$
which is a generalization of decimal or binary ($\beta=10,2$) 
representations by an arbitrary base $\beta>1$. 
%Hereafter we assume that $\beta\not \in \N$. 
The word $d_{\beta}(x):=x_1x_2\dots $ corresponding to $x$ 
is an infinite word over $\A=\{0,1,\dots, \lceil \beta \rceil -1\}$.
An infinite word $x_1x_2\dots\in \A^{\N}$ is eventually periodic if it is written as $x_1x_2\dots x_m(x_{m+1}\dots x_{m+p})^{\infty}$. We choose the minimum 
$m$ 
and $p$ when such $m$ and $p$ exist and call it $(m,p)$-periodic.
The dynamical properties of a piecewise linear
map are governed by the orbit of discontinuities. For $T_{\beta}$, there exists
essentially only one discontinuity corresponding to the right end point $1$.
The expansion of $1$
is the word $d_{\beta}(1):=\lim_{\epsilon \downarrow 0} d_{\beta}(1-\epsilon)=c_1c_2\dots$.
For a fixed $\beta$, 
an element $x_1x_2\dots\in \A^{\N}$ is realized as $d_{\beta}(x)$ with some $x\in [0,1)$ if and only if 
$$
\sigma^n(x_1x_{2}\dots) \lexl c_1c_2\dots 
$$
for any $n\in \N$. 
Here the shift is defined as $\sigma((x_{i})):=(x_{i+1})$ 
for a (one-sided or two-sided) infinite word $(x_i)\in \A^{\N} \cup \A^{\Z}$
and $\lexl$ is the lexicographic order.
We say that $x_1x_2\dots\in \A^{\N}$ is admissible if this condition holds.
A finite word $x_1\dots x_n\in \A^*$ is admissible if $x_1\dots x_n0^{\infty}$ is admissible.
An element of  $c_1c_2\dots\in \A^{\N}$ is realized as $d_{\beta}(1)$ with some $\beta>1$ if and only if
\begin{equation}
\label{Selflex}
\sigma^n(c_1 c_{2}\dots) \lexl c_1c_2\dots
\end{equation}
for any $n\ge 1$, see \cite{Parry1960,Ito-Takahashi:74} for details.
Given $\beta>1$, the beta shift $(X_{\beta},\sigma)$ is a subshift 
consisting of the set of bi-infinite words $(a_i)_{i\in \Z}$ over $\A$ 
such that 
every subword $a_na_{n+1}\dots a_m$ is admissible.
Beta shift $X_{\beta}$ is sofic if and only if 
$d_{\beta}(1)$ is eventually periodic, 
and $X_{\beta}$ is a subshift of finite type if and only if $d_{\beta}(1)$ is purely periodic,
i.e., $d_{\beta}(1)$ is $(0,p)$-periodic, see
\cite{Blanchard1989,Akiyama_Dual}\footnote{In \cite{Parry1960}, expansion of one is defined formally by $(\lfloor \beta T_\beta^n(1) \rfloor)_{n\in \N}$ and 
the
purely periodic expansion $d_{\beta}(1)=(c_1c_2\dots c_{p-1}c_p)^{\infty}$ is expressed as a finite expansion $c_1c_2\dots c_{p-1}(1+c_p) 0^{\infty}$.}.
The $\beta$ is called a Parry number in the former case, and a simple Parry number in the latter case. When the topological dynamics $(X_{\beta},\sigma)$ is sofic, $T_{\beta}$  belongs to an important class of interval maps; Markov maps 
by finite partition (see \cite{BG, PY}).

A Pisot number is an algebraic integer $>1$ so
that all of whose conjugates have modulus less than one. A Salem number
is an algebraic integer $>1$ so
that all of whose conjugates have modulus not greater than one and at least one of the conjugates has modulus one. 
If $\beta$ is a Pisot number, then $\{ T_{\beta}^{i}(x)\ |\ i\in \N \}$ is finite for $x\in \Q(\beta)$, i.e., the word $d_{\beta}(x)$ is eventually periodic, see \cite{SCHMIDT1980,Bertrand1977}.
Consequently, a Pisot number is a Parry number. 
Schmidt \cite{SCHMIDT1980}
proved that if $d_{\beta}(x)$ is eventually periodic for every
$x\in \Q\cap [0,1)$ then $\beta$ is a Pisot or Salem number. 
Determining periodicity/non-periodicity 
of $d_{\beta}(x)$ by a Salem number $\beta$ and 
$x\in \Q(\beta)$ remains a difficult problem.
The main obstacle is that we do not have any idea until present
to show that $d_{\beta}(x)$ is {\bf not} eventually periodic when $\beta$ is a 
Salem number.
Boyd \cite{Boyd1989Degree4}
showed that a Salem number of degree $4$
is a Parry number, by classifying all shapes of $d_{\beta}(1)$.
Since then, apart from a computational or heuristic discussion like 
Boyd \cite{Boyd1996}, Hichri \cite{Hichri2015beta, hichri2014betapar,hichri2014beta},
we have very few results
on the beta expansion by Salem numbers.  
In this paper, we make some additions to this direction. 

\begin{thm}
\label{main}
For a Salem number $\beta$ of degree $d$, there exist infinitely many positive integer $m$'s that $\beta^m$ is a Parry number. 
More precisely, there exists a positive
constant $c(d)$ depending only on $d$ that
$$
\liminf_{M\to \infty} \frac {\mathrm{Card}\left\{ m\in [1,M]\cap \Z\ \left.|\ d_{\beta^m}(1) \text{ is $(1,p)$-periodic with some $p\in \N$} \right.\right\}}
M\ge c(d).
$$
\end{thm}

Note that $c(4)=1$ was shown in \cite{Boyd1989Degree4}.
Our method gives a rather small bound $c(d)=(3d)^{-d}$, 
see Remark \ref{Cd}.
We can give a good lower bound when $d=6$.

\begin{thm}
\label{main2}
Given a sextic Salem number $\beta$, for more than half of positive integer $m$'s, $\beta^m$ is a $(1,p)$-periodic Parry number for some $p\in \N$.
\end{thm}

Finally, we discuss an interesting connection to four-dimensional discretized rotation.

\section{Preliminary}
We review basic results on Salem numbers (c.f. \cite{M.J.Bertin1992,smyth2014seventy}).
It is easy to show that a Salem number has even degree $2d \ge 4$ and its minimal polynomial is self-reciprocal. 
Thus $d_{\beta}(1)$ can not be purely periodic for a Salem number $\beta$, see \cite{Akiyama:98}.
Let $P(x)\in \Z[x]$ 
be a monic irreducible self-reciprocal polynomial of even degree $2d$. Putting 
$Q(x):=P(x)/x^{2d}$, we have $Q(x)\in \R[x+x^{-1}]$.
We write $Q(x)=G(y)$ with $y=x+x^{-1}$. Then $P(x)$ is a minimum polynomial
of a Salem number if and only if $G(2)<0$ and $G(y)$ has $d-1$ distinct roots in $(-2,2)$. $G$ is coined trace polynomial of $P$ in \cite{Boyd1996}.
The factorization
$$
G(y)=(y-\gamma)(y-\alpha_1)\dots (y-\alpha_{d-1})$$ 
with $\gamma>2$ and $\alpha_i\in (-2,2)$ corresponds to
the factorization of $P(x)$ in $\R[x]$:
\begin{equation}
\label{FactorSalem}
P(x)=\left(x-\beta\right)\left(x-\frac 1{\beta}\right) \prod_{i=1}^{d-1} \left(x^2+\alpha_i x+1\right)
\end{equation}
where $\gamma=\beta+1/\beta$ and $x^2+\alpha_i x+1$ gives a root 
$\exp(\theta_i \sqrt{-1}) =\cos(\theta_i)\pm \sin(\theta_i)\sqrt{-1}$ of $P(x)$
 with $\alpha_i=-2\cos(\theta_i)$ and $\theta_i\in (0,\pi)$.
It is well known that $1, \theta_1/\pi, \theta_2/\pi, \dots, \theta_{d-1}/\pi$ 
are linearly independent over $\Q$, i.e., $\exp(\theta_i\sqrt{-1})\ (i=1,\dots, d-1)$ are multiplicatively independent.
This is shown by applying a conjugate map to the possible
multiplicative relation among them, c.f. \cite{DubJan:20}. 
Note that this fact
guarantees that $\beta^m\ (m=1,2,\dots)$ are Salem numbers of 
the same degree $2d$.
Applying this linear independence, we see that
\begin{equation}
\label{TorusRot}
\left(\frac{m \theta_1}{2\pi}, \dots, \frac{m \theta_n}{2\pi}\right) \bmod{\Z^n}\end{equation}
is uniformly distributed in $(\R/\Z)^n$, that is, for any parallelepiped
$$
I=[a_1,b_1]\times [a_2,b_2]\times \dots \times [a_n,b_n]
$$
we have
$$
\lim_{M\to \infty} \frac 1M \sum_{m=1}^M \chi_I\left( \left(\frac{m \theta_1}{2\pi}, \dots, \frac{m \theta_n}{2\pi}\right) \bmod{\Z^n}\right) = \prod_{i=1}^n (b_i-a_i)
$$
where $\chi_I$ is the characteristic function of $I$.
Indeed, this is shown by the higher dimensional Weyl criterion (\cite{Kuipers-Niederreiter:74}).
Note that it is understood as unique ergodicity of the action
$$
x\mapsto x + \left(\frac{\theta_1}{2\pi}, \dots, \frac{\theta_n}{2\pi}\right)
$$
on $(\R/\Z)^n$. Since $(\R/\Z)^n$ is a compact metric group, minimality and
unique ergodicity is equivalent (\cite{Walters:82}), and minimality is a little
easier to show.
%In \cite{Akiyama-Tanigawa:04},
%the distribution of $\beta^n \pmod{\Z}$ is discussed using this property.

\section{Our strategy}

Our idea is to find a nice region so
that if complex conjugates of a Salem number $\beta>1$ 
fall into this region
and $\beta$ is sufficiently large, then $d_{\beta}(1)$ is eventually periodic.
We realize this idea in a general form 
which can be applied to the dominant real root of self-reciprocal polynomials, 
not only Salem numbers. The statement seems 
useless for a single $\beta$ (because it is easier to compute $d_{\beta}(1)$ directly) but we will find a nice application in the following sections.

\begin{lem}
\label{Str1}
Let us fix constants $u,v$ with $0<u<v<1$. 
If a monic polynomial $f(x)\in \Z[x]$ of degree $2n+2$
satisfies
$$
f(\beta)=0,\ f(0)=1, \ \beta> \max\left\{ \frac 2u, \frac 1{1-v} \right\}
$$
and
$$
\frac{f(x)}{(x-\beta)(x-1/\beta)} = x^{2n}+1+\sum_{i=1}^{2n-1}g_i x^i
$$
with $u<g_i<v$ for $i=1,\dots,2n-1$, then $f(x)$ is self-reciprocal and 
$d_{\beta}(1)$ is $(1,2n+1)$-periodic.
\end{lem}

\begin{proof}
Set $g_0=g_{2n}=1$. Putting
$$
f(x)=x^{2n+2}+1-\sum_{i=1}^{2n+1} c_i x^i,
$$
we obtain
$$
c_1=\beta+1/\beta-g_1,\quad
c_{2n+1}=\beta+1/\beta-g_{2n-1}
$$
and
$$
c_i=(\beta+1/\beta) g_{i-1} - g_{i-2}-g_{i}
$$
for $i=2,\dots, 2n$. 
Since $c_i\in \Z$, $g_1$ and $g_{2n-1}$ are uniquely
determined by $\beta+1/\beta$ and thus $g_1=g_{2n-1}$ and $c_1=c_{2n+1}$.
Moreover by induction, we see $c_{i}=c_{2n+2-i}$ for every $i=1,\dots, 2n+1$, i.e., $f(x)$ is self-reciprocal. 
By assumption, we have
$$
c_i\ge \beta u -2 >0,\quad
c_1-c_i\ge (\beta+1/\beta)(1-v) -1 >0
$$
for $i=2,\dots,2n$. Therefore, we have
$$
c_i\in \Z, \quad c_1> c_j> 0\ (j=2,\dots, 2n).
$$
One can write 
$$
\beta^{2n+2}+1-\sum_{i=1}^{2n+1} c_i \beta^i=0
$$
as a representation of zero in base $\beta$:
$$
(-1),c_{1},c_{2},\dots, c_{n},c_{n+1},c_{n}, \dots ,c_2, c_1, (-1).
$$
Adding its $2n+1$ shifted form, we see
$$
(-1),c_{1},c_{2},\dots, c_{n+1}, \dots ,
c_2, c_1-1, c_1-1, c_2, \dots, c_{n+1}, \dots ,c_2, c_1,(-1)
$$
is another representation of zero. Iterating this shifted addition, 
we obtain an infinite expansion of $1$:
$$
c_1(c_2,c_3,\dots, c_{n},c_{n+1},c_{n},\dots, c_2,c_1-1,c_1-1)^{\infty}
$$
which satisfies\footnote{This implies $c_1<\lceil \beta \rceil-1$ as well.} the lexicographic condition (\ref{Selflex}).
Therefore $d_{\beta}(1)$ is $(1,2n+1)$-periodic.
\end{proof}

Logically we have to fix $u,v$ at first. However, 
this is often impractical.
When we apply Lemma \ref{Str1} to a Salem number $\beta$ 
of degree $2d$, we take a self-reciprocal monic polynomial $R(x)\in \Z[x]$
and study $f(x)=R(x)P(x)$ of degree $2n+2$, 
where $P(x)$ is the minimum polynomial of $\beta$.
Note that one can take $R(x)=1$ as well.
We check if there exist $u$ and $v$ which satisfy our requirements.
Since
$$
\frac{f(x)}{(x-\beta)(x-1/\beta)}= \sum_{i=0}^{2n} g_i x^i=
R(x) \prod_{i=1}^{d-1} (x^2+\alpha_i x+1)
$$
with $\alpha_i\in (-2,2)$, for a fixed $R(x)$, 
the problem is reduced to the set of solutions 
$(\alpha_1,\dots,\alpha_{d-1})\in \R^{d-1}$
for the system of inequalities
$$
0<g_i<1 \ (i=1,2,\dots, 2n-1),\quad -2<\alpha_j<2\ (j=1,\dots,d-1),
$$
over $d-1$ variables $\alpha_1,\dots,\alpha_{d-1}$.
A self-reciprocal monic polynomial $R(x)$ gives a certain choice of $u,v$ if and only if 
the set of solutions contains an inner point in the space $\R^{d-1}$, and
for every inner point we can find $u,v$.
Solving this set of inequalities is not an easy task 
in general, but it is feasible for degree $6$ 
since all the inequalities are quadratic.
In \S 4, we use this method to find good regions for $\beta$.
It is also useful to solve a slightly wider set of inequalities:
\begin{equation}
\label{DefIneq}
0<g_i\le 1 \ (i=1,2,\dots, 2n-1),\quad -2<\alpha_j<2\ (j=1,\dots,d-1).
\end{equation}
The solution may contain the case $g_i=1$. 
The proof of Lemma \ref{Str1} works in the same way under a little
more involved assumptions and we have a

\begin{lem}
\label{Str2}
Let us fix constants $u,v$ with $0<u<v<1$. 
If a monic polynomial $f(x)\in \Z[x]$ of degree $2n+2$
satisfies
$$
f(\beta)=0,\ f(0)=1, \ \beta> \max\left\{ \frac 2u, \frac 1{1-v} \right\}
$$
and
$$
\frac{f(x)}{(x-\beta)(x-1/\beta)} = x^{2n}+1+\sum_{i=1}^{2n-1}g_i x^i
$$
with
$$
u<g_i<v
$$
or
$$
\left(g_i=1\ \text{and}\ g_1<g_{i-1}+g_{i+1}\right)
$$
or
$$
\left(
g_i=1\ \text{and}\ g_1=g_{i-1}+g_{i+1}\  \text{and} \ g_1-g_{i+1}>u\right)
$$
hold for $i=1,\dots,2n-1$, then $f(x)$ is self-reciprocal and 
$d_{\beta}(1)$ is $(1,2n+1)$-periodic.
\end{lem}

\begin{proof}
If $g_i=1\ \text{and}\ g_1<g_{i-1}+g_{i+1}$ holds, 
then $c_1-c_{i+1}=g_{i-1}+g_{i+1}-g_1>0$.
If $g_i=1\ \text{and}\ g_1=g_{i-1}+g_{i+1}\  \text{and} \ g_1-g_{i+1}>u$ holds,
then $c_1=c_{i+1}$ and $c_2-c_{i+2}=(\beta+1/\beta)(g_1-g_{i+1})-1-g_2 +g_{i}+g_{i+2}>\beta u-2>0$. 
Therefore the lexicographic condition (\ref{Selflex}) holds
in all cases.
\end{proof}

We shall see later that the
effect of this small extension of Lemma \ref{Str1} is pretty large both in theory and in practice. Indeed, Lemma \ref{Str2} works very
well with discretized rotation algorithm, see the discussion in \S 5.

\section{Proof of Theorem \ref{main}}
\label{Proof1}

We start with a 

\begin{lem}
\label{Circle}
For a positive even integer $2n$, all roots of the self-reciprocal polynomial
$$
P(x)=x^{2n}+ d_{2n-1} x^{2n-1}+\dots + d_1 x +1 \quad  (d_{2n-i}=d_i)
$$
with $d_i\in \R$ and $|d_i|<1/(2n-2)$ are on the unit circle. 
\end{lem}

\begin{proof}
Let $G_0(y):=(x^{2n}+1)/x^{n}$ and $G_1(y):=P(x)/x^{n}$ with $y=x+x^{-1}$. 
Since $$G_0(2\cos(\pi k /n))=2(-1)^k$$ for $k=0,1,2,\dots, n$, 
$G_0(y)$ has $n$ real roots $\psi_i\ (i=1,2,\dots, n)$ with
\begin{align*}
2&>\psi_1>2\cos\left(\frac{\pi}n\right)>\psi_2>2\cos\left(\frac{2\pi}n\right)>\psi_3>\dots\\
&\dots> 2\cos\left(\frac{(n-2)\pi}n\right)>\psi_{n-1}>
2\cos\left(\frac{(n-1)\pi}n\right)>\psi_n>-2.
\end{align*}
From $|d_i|<1/(2n-2)$, we have $G_1(2\cos(\pi k /n))<0$ for odd $k$ and
$G_1(2\cos(\pi k /n))>0$ for even $k$.
By intermediate value theorem, 
$G_1(y)$ has $n$ real roots $\psi'_i\ (i=1,\dots,n)$ in $(-2,2)$
satisfying the same inequality as $\psi_i$. Therefore
we have
$$G_1(y)=\prod_{i=1}^{n} (y-\psi'_i).$$
Coming back to $P(x)$ we get the assertion.
\end{proof}

For degree 4, 
we have nothing to do since Boyd \cite{Boyd1989Degree4} showed that every Salem number of degree 4 is a $(1,p)$-periodic Parry number for some $p\in \N$.
Consider a polynomial
$$
h(x)=x^{2n}+ \frac 1{4(n-1)}\sum_{i=1}^{2n-1} x^i +1.
$$
By Lemma \ref{Circle}, we have
$$
h(x)=\prod_{i=1}^{n} \left(x-\exp(\eta_i \sqrt{-1})\right)\left(x-\exp(-\eta_i \sqrt{-1})\right)$$
with 
\begin{align}
\label{Interlace}
2&>2\cos(\eta_1)>2\cos\left(\frac{\pi}n\right)>2\cos(\eta_2)>2\cos\left(\frac{2\pi}n\right)>2\cos(\eta_3)>\dots\\ 
&\dots> 2\cos\left(\frac{(n-2)\pi}n\right)>2\cos(\eta_{n-1})>
2\cos\left(\frac{(n-1)\pi}n\right)>2\cos(\eta_{n})>-2.
\nonumber
\end{align}
This type of discussion is called ``interlacing" and efficiently
used in the construction
of Salem numbers having desired properties, see \cite{Smyth2000,Smyth2005,Akiyama-Kwon:08} and its references.

Considering coefficients as a continuous function of roots,
there exists a constant $\varepsilon>0$ that if
$$
\psi_i \in [\eta_i-\varepsilon, \eta_i+\varepsilon]
$$
for $i=1,2,\dots, n$, then we have
\begin{equation}
\label{Key}
\prod_{i=1}^{n} \left(x-\exp(\psi_i \sqrt{-1})\right)\left(x-\exp(-\psi_i \sqrt{-1})\right)
=x^{2n}+1+\sum_{i=1}^{2n-1} g_i x^i 
\end{equation}
with $\frac {1}{6(n-1)}<g_i<\frac {1}{3(n-1)}$. See Remark \ref{Cd} for the 
choice of $\varepsilon$.

Let $\beta$ be a Salem number of degree $2n+2$ with $n\ge 2$ and let
$\theta_i\in (0,\pi)\ (i=1,\dots,n)$ 
be the arguments of the conjugates of $\beta$ 
on the unit circle determined as in (\ref{FactorSalem}).
Since $1, \theta_1/\pi,\dots, \theta_{n}/\pi$ are linearly independent over $\Q$,
by Kronecker's approximation Theorem, we find
infinitely many positive integer $m$'s such that
\begin{equation}
\label{Rot}
\frac{m \theta_i}{2\pi} \pmod{\Z} \in [\eta_i-\varepsilon, \eta_i+\varepsilon]
\end{equation}
hold for $i=1,2,\dots,n$.
For an integer $m$ with this property, 
the minimum polynomial of $\beta^m$ has the form
$$
\left(x-\beta^m\right)\left(x-\frac 1{\beta^m}\right)\left(
x^{2n}+1+\sum_{i=1}^{2n-1} g^{(m)}_i x^i \right)
$$
and 
$\frac {1}{6(n-1)}<g^{(m)}_i <\frac {1}{3(n-1)}$
for $i=1,2,\dots, 2n-1$. By Lemma \ref{Str1}, we see that
$d_{\beta^m}(1)$ is $(1,2n+1)$-periodic for sufficiently large $m$.
Finally, we show that
$$
\{m\in \N\ | \ d_{\beta^m}(1) \text{ is $(1,2n+1)$-periodic} \}
$$
has positive lower natural density.
Using the fact that (\ref{TorusRot})
is uniformly distributed in $(\R/\Z)^n$, we have 
$$
\liminf_{M\to \infty} \frac 1M \mathrm{Card}\left\{ m\in [1,M]\cap \N\ 
|\ d_{\beta^m}(1) \text{ is $(1,2n+1)$-periodic} \right\} \ge (2\varepsilon)^n
$$
in view of (\ref{Rot}), giving the lower bound of the natural density.
\qed

\begin{rem}
\label{Cd}
We give a lower bound of $c(d)$ from the above proof.
First, we compute the 
asymptotic expansion of $\eta_i$ with respect to $n$, when $\eta_i$ is close to
$\pi/2$ (c.f. \cite{Akiyama:16}).
If $n$ is odd, then
we have
$$
\eta_{\lfloor n/2 \rfloor}= \frac{\pi}2+ \frac {1}{8n^2}+\frac {9}{64n^3}+ O\left(\frac 1{n^4}\right)
$$
which leads to $\tan(\eta_{\lfloor n/2 \rfloor})=-8 n^2+O(n)$.
Further, we have
$$
\eta_{\lfloor n/2 \rfloor\pm 1}= \frac{\pi}2 \pm \frac{\pi}{n}+ \frac {1}{8n^2}+\frac {9\mp 8\pi}{64n^3}+ O\left(\frac 1{n^4}\right)
$$
and
$$
\tan(\eta_i)=O(n)
$$
for $|i-\lfloor n/2\rfloor|>1$ in light of (\ref{Interlace}).
If $n$ is even, then we have 
$$
\eta_{n/2 \pm 1}= \frac{\pi}2 \pm \frac{\pi}{2n}+ \frac {1}{8n^2}+\frac {9\mp 4\pi}{64n^3}+ O\left(\frac 1{n^4}\right).
$$
Thus $\tan(\eta_i)=O(n)$ is valid
for all $i$ by (\ref{Interlace}).
%Therefore there exists a positive constant $\tau$ that 
%$\sum_{i=1}^n |\tan(\eta_i)|\le \tau n^2$ holds regardless of the parity of 
%$n$.

Second, comparing the coefficients of $\prod_{i=1}^n (x^2+2\cos(\eta_i+\nu) x+ 1)$
and those of $\prod_{i=1}^n (x^2+2\cos(\eta_i) x+ 1)$, 
we see that the desired inequality of (\ref{Key}) holds if
\begin{equation}
\label{ProdEst}
\frac 23<\prod_{i=1}^n \left|\frac{\cos(\eta_i+\nu)}{\cos(\eta_i)}\right| < \frac 43
\end{equation}
for $|\nu|<\eps$.
Using an easy inequality
$$
1-|\tan(x)y|-\frac{y^2}2<
\left|\frac{\cos(x+y)}{\cos(x)}\right|<1+|\tan(x) y|
$$
for $\cos(x)\neq 0$ and above estimates on $\eta_i$, 
we see that
there exists a positive constant $\kappa$ so that
if $|\eps|<\kappa/n^2$, then (\ref{ProdEst}) holds. Thus
we obtain $c(d)\ge (\kappa/n^{2})^{n}$ with $d=2n+2$. Making explicit 
the implied constants of Landau symbols, we see that $\kappa=1/32$ suffices.
Thus we have $c(d)\ge (32(d/2)^2)^{-d/2}>(3d)^{-d}$.
\end{rem}

%\begin{rem}
%\label{Method1}
%Our method works similarly for a point
%$(\theta_1,\dots,\theta_n)\in (0,\pi)^n$
%with $0\le g_i<1$ for $i=1,\dots,2n-1$ in (\ref{Key}).
%The essence of the above proof is to show 
%that the set of such $(\theta_1,\dots,\theta_n)$'s 
%contains an inner point.
%\end{rem}

\section{Proof of Theorem \ref{main2}}
\label{Proof2}
For the rest of this paper, we deal with a Salem number $\beta$ 
of degree 6. For simplicity of presentation, 
we prove that $c(6)\ge 0.458$. To show 
that $c(6)$ exceeds $0.5$, 
we have to use more polynomials and the computation 
becomes much harder, see Appendix 1.

To apply the discussion after
Lemma \ref{Str1}, there exists an efficient practical 
way to find $R(x)$
which is similar to the shift radix system
(c.f. \cite{Akiyama-Borbely-Brunotte-Pethoe-Thuswaldner:04}).
Pick a random $(\alpha_1,\alpha_2)\in(-2,2)^2$ and start with a coefficient vector $(1,c_1,c_2,c_1,1)$ of
$$
(x^2+\alpha_1 x+1)(x^2+\alpha_2 x+1)=x^4+c_1x^3+c_2 x^2+c_1 x+1,
$$
i.e., $c_1=\alpha_1+\alpha_2$, $c_2=\alpha_1 \alpha_2 +2$.
We wish to iterate the shifted addition like
\begin{align*}
(1,c_1,c_2,c_1,1) & \rightarrow
(1,c_1,c_2,c_1,1,0)+k (0, 1,c_1,c_2,c_1,1)\\
&=(1, c_1+k, c_2+c_1 k, c_1+c_2 k, 1+c_1 k ,k)
\end{align*}
with $k=-\lceil c_1-1\rceil$, to find a longer coefficient vector that all entries except the first and the last ones fall in $(0,1]$
as in Lemma \ref{Str2}. To make this idea into an algorithm, set 
$z_{-1}=(0,0,0,0)$ and for $z_n=(t_n(1),t_n(2),t_n(3),t_n(4))$, we define 
\begin{equation}
\label{ShiftedAddition}
z_{n+1}:=(t_n(2)+ c_1k_{n+1}, t_n(3)+ c_2k_{n+1}, t_n(4)+c_1 k_{n+1}, k_{n+1})
\end{equation}
with $k_{n+1}=-\lceil t_n(1)-1 \rceil$. Thus we see
$k_0=1$ and $z_0=(c_1,c_2,c_1,1)$.
We stop this iteration when $0<t_n(i)\le 1$ for $i=1,2,3$ and $t_n(4)=1$ (i.e. $k_{n}=1$) and 
obtain a candidate $R(x)=\sum_{i=0}^{n} k_i x^i$. 
It is natural to set $k_{n+1}=k_{n+2}=k_{n+3}=k_{n+4}=0$ and we obtain
$$
(x^4+c_1x^3+c_2x^2+c_1x+1)R(x)=\sum_{i=0}^{n+4} g_i x^i
$$
with
$g_0=g_{n+4}=1$ and $g_i=t_{i-1}(1)+k_i\in (0,1]\ (i=1,\dots,n+3)$. 
If $n$ became larger than a given threshold, then we restart with a different
$(\alpha_1,\alpha_2)$.
Applying this random search, we find polynomials $R_i(x)\ (i=1,\dots, 18)$ so
that Lemma \ref{Str1} give relatively large regions $\RR_i\ (i=1,\dots,18)$
defined by (\ref{DefIneq}).
\begin{align*}
R_1=&1,\\
R_2=&1+x,\\
R_3=&1+x^2,\\
R_4=&1+2 x+x^2,\\
R_5=&1+x+x^2+x^3,\\
R_6=&1+2 x+2 x^2+x^3,\\
R_7=&1-x+x^2-x^3+x^4,\\
R_8=&1+2 x+2 x^2+2 x^3+2x^4+x^5,\\
R_9=&1+x-x^2-x^3+x^4+x^5,\\
R_{10}=&1+2 x^2+2 x^4+x^6,\\
R_{11}=&1+x+x^2+2 x^3+x^4+x^5+x^6,\\
R_{12}=&1+x+x^2+2 x^3+2 x^4+x^5+x^6+x^7,\\
R_{13}=&1+x+2 x^2+2 x^3+2 x^4+2 x^5+x^6+x^7,\\
R_{14}=&1+2 x+x^2-x^3-x^4+x^5+2 x^6+x^7, \\
R_{15}=&1+x^2-x^3+x^4-x^5+x^6+x^8,\\
R_{16}=&1-2x+2x^2-x^3+x^4-2x^5+3x^6-2x^7+x^8-x^9+2x^{10}-2x^{11}+x^{12},\\
R_{17}=&1+3x+4x^2+3x^3+x^4,\\
R_{18}=&1+3x+4x^2+2x^3-2x^4-4x^5-2x^6+2x^7+4x^8+3x^9+x^{10}.
\end{align*}
%We can check that the corresponding 18 regions have non-empty interiors. 
For example,
$$
(x^4+c_1x^3+c_2x^2+c_1x+1)R_8=x^9+(c_1+2)x^8+\dots+(c_1+2)x+1
$$
gives the coefficient vector
$$
(1,c_1+2,2 c_1+c_2+2,3 c_1+2 c_2+2,4 c_1+2 c_2+3,
4 c_1+2 c_2+3,3 c_1+2 c_2+2,2 c_1+c_2+2,c_1+2,1)
$$
which gives rise to a system of linear inequalities
\begin{align*}
0<c_1+2\leq 1,\ 0<2 c_1+c_2+2\leq 1,\ 0<3 c_1+2 c_2+2\leq 1,\ 0<4 c_1+2 c_2+3\leq 1.
\end{align*}
Solving this system, we obtain a triangular region:
\begin{align*}
\bullet&\  c_1 >\frac{-2 c_2-3}{4} &\text{if}\qquad &c_2\in \left(\frac 12,\frac 52\right),\\
\bullet&\  c_1 \leq -1 &\text{if}\qquad &c_2\in \Big(\frac{1}{2}, 1\Big],\\
\bullet&\ c_1\leq \frac{-2 c_2-1}{3}& \text{if}\qquad &c_2\in\left(1,\frac{5}{2}\right).
\end{align*}
This collection of three sentences with $\bullet$ reads $c_2$ 
must be in at least one of the intervals, and we
take logical ``and" of the three.
Replacing $(c_1,c_2)$ by
$(\alpha_1+\alpha_2,\alpha_1\alpha_2+2)$, we can confirm that
$-2<\alpha_i<2$ for $i=1,2$ hold
in this triangle.\footnote{For a general $R(x)$, we compute the intersection with the region $-2<\alpha_i<2\ (i=1,2)$.
If this intersection is empty, then we have to 
restart with a different $(\alpha_1,\alpha_2)$.}
Thus we find the two curvilinear triangles $\RR_{8}$ in Figure \ref{SalemParry} bounded by segments and hyperbola.
We also found polynomials $L_i\ (i=1,\dots, 5)$ giving large regions $\LL_i\ (i=1,\dots, 5)$
where Lemma \ref{Str1} does not apply and Lemma \ref{Str2} is necessary.
\begin{align*}
L_1=&1-x^2+x^3+x^4-x^5+x^7,\\
L_2=&1-x+x^3-x^5+x^6, \\
L_3=&1-x^2+x^3+x^6-x^7+x^9,\\
L_4=&1-2x+2x^2-2x^4+3x^5-2x^6+2x^8-2x^9+x^{10},\\
L_5=&1-x+x^3-x^6+x^7+x^8-x^9+x^{12}-x^{14}+x^{15}.
\end{align*}
For example, the coefficients of $x^3$ and $x^8$ in
\begin{align*}
&(x^4+c_1 x^3+c_2x^2+c_1 x+1)L_1=\\
&\quad 1+c_1 x-x^2+c_2 x^2+x^3+2 x^4+c_1 x^4-c_2 x^4-x^5+c_2 x^5-x^6+c_2 x^6\\
&\quad +2 x^7 +c_1 x^7-c_2 x^7+x^8-x^9+c_2 x^9+c_1 x^{10}+x^{11}
\end{align*}
are equal to $1$, we have to use Lemma \ref{Str2}.

%This random search is more efficiently executed if 
%we exclude points $(\alpha_1,\alpha_2)$ in the found period cells 
%from the next trials.
%Moreover in this way, we have better chances 
%that the obtained period cells are mutually disjoint.
%Indeed 

We can check directly
that those 23 sets are mutually disjoint. 
We will see later in \S \ref{Four} that Lemma \ref{Str2} works fine
for any period and this disjointness is natural. Indeed,
if two period cells share an inner point, 
then their periods of the
corresponding discretized rotations must coincide, since $d_{\beta^m}(1)$ for a Salem number $\beta$
is uniquely determined by $m$.
This implies that the period cells must be identical.
See Figure \ref{SalemParry} and the explicit computation of period cells below.

\begin{figure}[htb]
\includegraphics[width=\columnwidth]{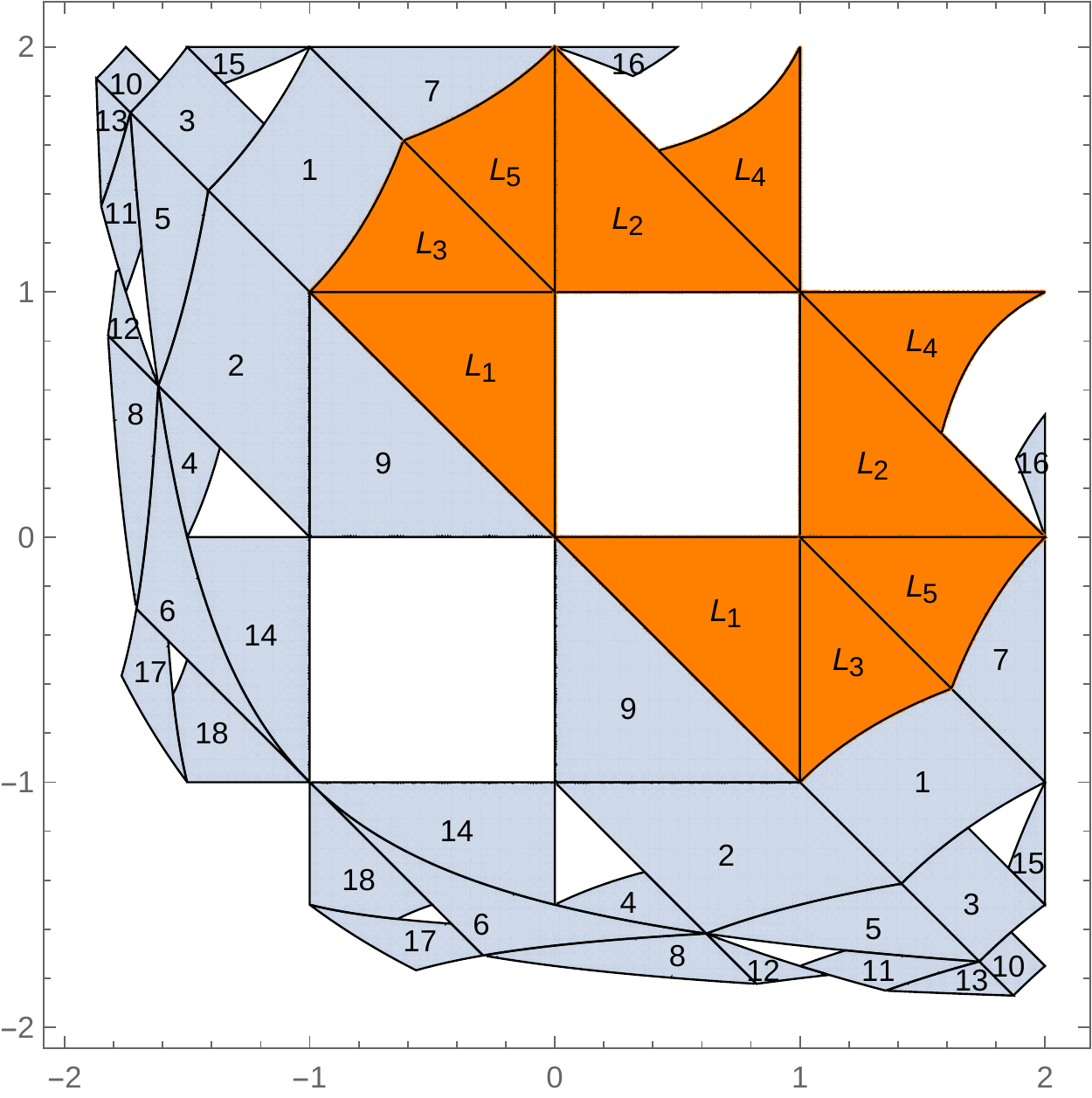}
\caption{Period cells in $(\alpha_1,\alpha_2)$ coordinate \label{SalemParry}}
\end{figure}

Choose a small $\eps>0$  
and consider the subset $\RR_i(\eps)$ of solutions  of the system of inequalities
$$
\eps<g_i<1-\eps \ (i=1,2,\dots, 2n-1),\quad -2<\alpha_j<2\ (j=1,\dots,d-1)
$$
for a polynomial $R_i\ (i=1,\dots, 18)$.
We also define $\LL_i(\eps)$ by  the system of inequalities
$$
\eps<g_i<1-\eps \ (i\in \{j \in [1,2n-1]\ |\ g_j\not = 1\}),
\quad -2<\alpha_j<2\ (j=1,\dots,d-1)
$$
for a polynomial $L_i\ (i=1,\dots,5)$. Since $g_i\in \Z+(\alpha_1+\alpha_2)\Z+\alpha_1\alpha_2\Z$,
$g_i$ is a constant 
if and only if $g_i=1$. Therefore unless $g_i=1$, 
the boundary equalities $g_i=\eps$ and $g_i=1-\eps$ give 
one-parameter families of (linear or hyperbolic) curves, which continuously
move along $\eps$.
In light of Lemma \ref{Str1} and \ref{Str2}, there exist $m_0=m_0(\eps)\in \N$ that 
if $m\ge m_0$ and $(-2\cos(m\theta_1/2\pi),-2\cos(m\theta_2/2\pi))$ falls into
$$
S(\eps):=\left(\bigcup_{i=1}^{18} \RR_i(\eps) \right) \cup 
\left(\bigcup_{i=1}^{5} \LL_i(\eps) \right),
$$
then $d_{\beta^m}(1)$ is $(1,p)$-periodic with some $p\in \N$.
Since $(m \theta_1/2\pi,m\theta_2/2\pi) \pmod{\Z^2}$ is uniformly distributed
and $S(\eps)$ is Jordan measurable,
the induced probability measure
is computed by
$$
\frac 1{\pi^2}\int\int_{S(\eps)} \frac {d\alpha_1 d\alpha_2}{\sqrt{(4-\alpha_1^2)(4-\alpha_2^2)}},
$$
since the normalized
Lebesgue measure $\frac{1}{2\pi}
d\theta$ on the unit circle
is projected 2 to 1 to the interval $[-2,2]$
by the map $\theta\mapsto -2\cos(\theta)=:\alpha$ and
$$\frac {2}{2\pi} d \theta
= \frac {1}{\pi}d(\arccos(-\alpha/2))
= \frac {1}{\pi}d(\pi-\arccos(\alpha/2))
%= -\frac {1}{\pi}d(\arccos(\alpha/2))
=\frac{1}{\pi}
\frac{d \alpha}{\sqrt{4-\alpha^2}}.
$$

Let $$S := \left(\bigcup_{i=1}^{18} \RR_i \right) \cup 
\left(\bigcup_{i=1}^{5} \LL_i \right)
$$
and $\mu$ be the 2-dimensional Lebesgue measure. Since both 
$S(\eps)$ and $S$ have piecewise smooth boundaries, 
we have
$$
\mu(S\setminus S(\eps))\to 0
$$
as $\eps\to 0$. Thus the above measure converges to 
$$
\frac 1{\pi^2}\int\int_S \frac {d\alpha_1 d\alpha_2}{\sqrt{(4-\alpha_1^2)(4-\alpha_2^2)}}\approx 0.458895
$$
which proves our theorem.
The explicit forms of the regions 
$\RR_i\ (i=1,2,\dots,18)$ and $\LL_i\ (i=1,2,\dots,5)$ are listed below, 
using the coordinate $(x,y)=(\alpha_1,\alpha_2)$ with $y<x$. Each collection of
sentences with $\bullet$ is read in the similar way as before.
%Since we are not interested in belongingness of the boundary to period cells, 
%we only desribe their interiors in this list.

\begin{itemize}
\item $\RR_{1}$:
\begin{itemize}

\item $y \leq -\frac{1}{x},\quad \mathrm{if}\  x\in\left[1,\frac{1+\sqrt{5}}{2}\right]$
\item $y\leq 1-x, \quad \mathrm{if}\  x\in \left[\frac{1+\sqrt{5}}{2},2\right] $
\item $y>-x, \quad \mathrm{if}\  x\in \left[1,\sqrt{2}\right] $
\item $y>-\frac{2}{x}, \quad \mathrm{if}\  x\in \left[\sqrt2,2\right]$

\end{itemize}
\item $\RR_{2}$:

\begin{itemize}
\item $y\leq -1, \quad \mathrm{if}\   x\in\left[0,1\right]$
\item $y\leq -x, \quad \mathrm{if}\  x\in \left[1,\sqrt{2}\right] $
\item $y>-1-x, \quad \mathrm{if}\  x\in \left[0,\frac{\sqrt{5}-1}{2}\right] $
\item $y>-\frac{2+x}{1+x}, \quad \mathrm{if}\  x\in \left[\frac{\sqrt{5}-1}{2},\sqrt{2}\right]$
\end{itemize}
\item $\RR_{3}$:

\begin{itemize}
\item $y\leq \frac{-2}{x}, \quad \mathrm{if}\   x\in\left[\sqrt{2},\frac{1+\sqrt{33}}{4}\right]$
\item $y\leq \frac{1}{2}-x, \quad \mathrm{if}\  x\in \left[\frac{1+\sqrt{33}}{4},2\right] $
\item $y>-x, \quad \mathrm{if}\  x\in \left[\sqrt{2},\sqrt{3}\right] $
\item $y>\frac{-3}{x}, \quad \mathrm{if}\  x\in \left[\sqrt{3},2\right]$
\end{itemize}
 
\item $\RR_{4}$:

\begin{itemize}
\item $y\leq -\frac{3+2x}{2+2x}, \quad \mathrm{if}\   x\in\left[0,\frac{\sqrt{3}-1}{2}\right]$
\item $y\leq -1-x, \quad \mathrm{if}\  x\in \left[\frac{\sqrt{3}-1}{2},\frac{\sqrt{5}-1}{2}\right] $

\item $y>-\frac{3+2x}{2+x}, \quad \mathrm{if}\  x\in \left[0,\frac{\sqrt{5}-1}{2}\right]$
\end{itemize}

\item $\RR_{5}$:

\begin{itemize}
\item $y\leq -\frac{2+x}{1+x}, \quad \mathrm{if}\   x\in\left[\frac{\sqrt{5}-1}{2},\sqrt{2}\right]$
\item $y\leq -x, \quad \mathrm{if}\  x\in \left[\sqrt{2},\sqrt{3}\right] $

\item $y>-\frac{3+2x}{2+x}, \quad \mathrm{if}\  x\in \left[\frac{\sqrt{5}-1}{2},\sqrt{3}\right]$

\end{itemize}

 \item $\RR_{6}$:

\begin{itemize}
\item $y\leq -\frac{3+2x}{2+x}, \quad \mathrm{if}\   x\in\left[-1,\frac{\sqrt{5}-1}{2}\right]$
\item $y>-2-x, \quad \mathrm{if}\  x\in \left[-1,\frac{\sqrt{2}-2}{2}\right] $

\item $y>-\frac{5+3x}{3+2x}, \quad \mathrm{if}\  x\in \left[\frac{\sqrt{2}-2}{2},\frac{\sqrt{5}-1}{2}\right]$
\end{itemize}

\item $\RR_{7}$:

\begin{itemize}
\item $1-x<y\leq \frac{2-x}{2+x}, \quad \mathrm{if}\   x\in \Big[\frac{\sqrt{5}+1}{2},2\Big)$
%\item $y>1-x, \quad \mathrm{if}\  x\in \Big[\frac{\sqrt{5}+1}{2},2\Big) $
\end{itemize}

\item $\RR_{8}$:

\begin{itemize}
\item $y\leq -\frac{5+3x}{3+2x}, \quad \mathrm{if}\  x\in \left[\frac{\sqrt{2}-2}{2},\frac{\sqrt{5}-1}{2}\right]$

\item $y\leq -1-x, \quad \mathrm{if}\  x\in \left[\frac{\sqrt{5}-1}{2},\frac{\sqrt{7}-1}{2}\right] $
\item $y>-\frac{7+4x}{4+2x}, \quad \mathrm{if}\   x\in\left[\frac{\sqrt{2}-2}{2},\frac{\sqrt{7}-1}{2}\right]$

\end{itemize}

\item $\RR_{9}$:

\begin{itemize}

\item $-1<y\leq -x, \quad \mathrm{if}\  x\in \left(0,1\right] $

%\item $y>-1, \quad \mathrm{if}\   x\in\left(0,1\right]$

\end{itemize}

\item $\RR_{10}$:

\begin{itemize}

\item $y\leq \frac{-3}{x}, \quad \mathrm{if}\   x\in\left[\sqrt{3},\frac{1+\sqrt{193}}{8}\right]$
\item $y\leq \frac{1}{4}-x, \quad \mathrm{if}\  x\in \left[\frac{1+\sqrt{193}}{8},2\right] $
\item $y>-x, \quad \mathrm{if}\  x\in \left[\sqrt{3},\sqrt{\frac{7}{2}}\right] $
\item $y>\frac{-7}{2x}, \quad \mathrm{if}\  x\in \left[\sqrt{\frac{7}{2}},2\right]$
\end{itemize}

\item $\RR_{11}$:

\begin{itemize}
\item $y\leq -\frac{5+2x}{2+2x}, \quad \mathrm{if}\   x\in\left[\frac{\sqrt{111}-3}{8},\frac{\sqrt{33}-1}{4}\right]$

\item $y\leq -\frac{3+2x}{2+x}, \quad \mathrm{if}\  x\in \left[\frac{\sqrt{33}-1}{4},\sqrt{3}\right]$

\item $y>-\frac{4+3x}{3+x}, \quad \mathrm{if}\   x\in\left[\frac{\sqrt{111}-3}{8},\frac{\sqrt{41}-1}{4}\right]$

\item $y>-\frac{3+x}{1+x}, \quad \mathrm{if}\  x\in \left[\frac{\sqrt{41}-1}{4},\sqrt{3}\right]$

\end{itemize}

 \item $\RR_{12}$:

\begin{itemize}

\item $y\leq -\frac{4+3x}{3+x}, \quad \mathrm{if}\  x\in \left[\frac{\sqrt{5}-1}{2},\frac{\sqrt{19}-1}{3}\right]$
\item $y>-1-x, \quad \mathrm{if}\  x\in \left[\frac{\sqrt{5}-1}{2},\frac{\sqrt{7}-1}{2}\right] $
\item $y>-\frac{6+3x}{3+2x}, \quad \mathrm{if}\   x\in\left[\frac{\sqrt{7}-1}{2},\frac{\sqrt{19}-1}{3}\right]$

\end{itemize}

 \item $\RR_{13}$:

\begin{itemize}

\item $y\leq -\frac{3+x}{1+x}, \quad \mathrm{if}\  x\in \left[\frac{\sqrt{41}-1}{4},\sqrt{3}\right]$

\item $y\leq -x, \quad \mathrm{if}\  x\in \left[\sqrt{3},\sqrt{\frac{7}{2}}\right] $
\item $y>-\frac{7+4x}{4+2x}, \quad \mathrm{if}\   x\in\left[\frac{\sqrt{41}-1}{4},\sqrt{\frac{7}{2}}\right]$

\end{itemize}

 \item $\RR_{14}$:

\begin{itemize}

\item $-\frac{3+2x}{2+x}<y\leq -1, \quad \mathrm{if}\   x\in \left[-1,0\right]$
%\item $y>-\frac{3+2x}{2+x}, \quad \mathrm{if}\  x\in \left[-1,0\right] $

\end{itemize}

 \item $\RR_{15}$:

\begin{itemize}

\item $\frac{1}{2}-x<y\leq \frac{3-x}{1-x}, \quad \mathrm{if}\  x\in \Big[\frac{\sqrt{41}+1}{4},2\Big) $

%\item $y>, \quad \mathrm{if}\   x\in \Big[\frac{\sqrt{41}+1}{4},2\Big)$

\end{itemize}

 \item $\RR_{16}$:

\begin{itemize}

\item $\frac{6-3x}{3-x}<y\leq \frac{7-4x}{4-3x}, \quad \mathrm{if}\   x\in \Big[\frac{11+\sqrt{61}}{10},2\Big)$

%\item $y>\frac{6-3x}{3-x}, \quad \mathrm{if}\  x\in \Big[\frac{11+\sqrt{61}}{10},2\Big)$

\end{itemize}

 \item $\RR_{17}$:

\begin{itemize}
\item $y\leq -\frac{8+5x}{5+3x}, \quad \mathrm{if}\   x\in\left[-1,\frac{\sqrt{2}-3}{3}\right]$

\item $y\leq -2-x, \quad \mathrm{if}\  x\in \left[\frac{\sqrt{2}-3}{3},\frac{\sqrt{2}-2}{2}\right]$

\item $y>-\frac{6+3x}{3+x}, \quad \mathrm{if}\  x\in \left[-1,\frac{\sqrt{13}-7}{6}\right]$
\item $y>-\frac{5+3x}{3+2x}, \quad \mathrm{if}\   x\in\left[\frac{\sqrt{13}-7}{6},\frac{\sqrt{2}-2}{2}\right]$

\end{itemize}

\item $\RR_{18}$:

\begin{itemize}

\item $y\leq -2-x, \quad \mathrm{if}\  x\in \left[-1,\frac{-1}{2}\right]$
\item $y>-\frac{8+5x}{5+3x}, \quad \mathrm{if}\  x\in \left[-1,\frac{\sqrt{21}-11}{10}\right]$
\item $y>-\frac{7+5x}{5+4x}, \quad \mathrm{if}\   x\in\left[\frac{\sqrt{21}-11}{10},\frac{-1}{2}\right]$

\end{itemize}

\item $\LL_1$:

\begin{itemize}
\item $-x<y\le 1, \quad \mathrm{if}\ x\in (0,1]$
\end{itemize}

\item $\LL_2$:

\begin{itemize}
\item $0<y\le 2-x, \quad \mathrm{if}\ x\in (1,2)$
\end{itemize}

\item $\LL_3$:

\begin{itemize}

%\item $y>-\frac{1}{x}, \quad \mathrm{if}\  x\in \Big(1,\frac{1+\sqrt{5}}{2}\Big]$
\item $-\frac{1}{x}<y\leq 1-x, \quad \mathrm{if}\   x\in\Big(1,\frac{1+\sqrt{5}}{2}\Big]$

\end{itemize}

 \item $\LL_4$:

\begin{itemize}

\item $y\leq 1, \quad \mathrm{if}\  x\in \left[1,2\right]$
\item $y>2-x, \quad \mathrm{if}\  x\in \left[1,\frac{3+\sqrt{2}}{3}\right]$
\item $y>\frac{6-4x}{4-3x}, \quad \mathrm{if}\   x\in\left[\frac{3+\sqrt{2}}{3},2\right]$

\end{itemize}

 \item $\LL_5$:

\begin{itemize}

\item $y\leq 0, \quad \mathrm{if}\  x\in \left[1,2\right]$

\item $y>1-x, \quad \mathrm{if}\  x\in \left[1,\frac{1+\sqrt{5}}{2}\right]$
\item $y>\frac{2-x}{1-x}, \quad \mathrm{if}\   x\in\left[\frac{1+\sqrt{5}}{2},2\right]$

\end{itemize}

\end{itemize}

\begin{rem}
By examining the beta expansion of  23899  Salem numbers of degree 6 and trace at most 19, 
there are 18250 (about $76\%$)  Salem numbers that satisfy 
$$
 - 2 < \alpha_1 < 0 < \alpha_2 < 2.
$$
They are  Parry  numbers with relatively small orbit size $(\max(m,p)<1000)$. 
%Since the probability for $\alpha_1 \alpha_2<0$ is $1/2$; which is computed as in the proof of Theorem \ref{main2}, 
% we guess $c(6)$ could be $1/2$. However,  because of the large diversity in  the values  $(m,p)$,
%it seem difficult to make this rigorous. 

A heavy computational effort may be required to substantially 
improve Theorem \ref{main2} (or Proposition \ref{epnsion of reggular} in Appendix 2).
For example,
$
-r<\alpha_1<0<\alpha_2<r<2<\gamma .
$  
with $r>1$ quite close to 1, the orbit size starts 
taking many different values. 
For instance, the following polynomials: 
 
 $x^6-7x^5-3x^4-11x^3-3x^2-7x+1$,
$x^6-9x^5-x^4-11x^3-x^2-9x+1$ and   $x^6-8x^5+10x^4-15x^3+10x^2-8x+1$ satisfy respectively: $\alpha_1\approx-1.08$,  $\alpha_1\approx-1.05$  and  $\alpha_2\approx 1.1$ but $(m,p)$ equals to $(6,23)$, $(6,35)$ and $(1,119)$  respectively.
\end{rem}

\section{Four dimensional discretized rotation}
\label{Four}

Substituting variables of the algorithm (\ref{ShiftedAddition}) in \S 4 by 
$$
\begin{pmatrix} t_n(1)\cr t_n(2)\cr t_n(3) \cr t_n(4)\end{pmatrix}
= \begin{pmatrix} 1& c_1 & c_2 &  c_1 \cr
                              0& 1 & c_1 & c_2\cr
                              0 & 0& 1 & c_1\cr
                            0  & 0 & 0 & 1\end{pmatrix}  
\begin{pmatrix} k_{n-3}\cr k_{n-2}\cr k_{n-1} \cr k_n\end{pmatrix},
$$
we obtain an integer sequence $(k_n)_{n\ge -4}$
which satisfies
\begin{equation}
\label{4SRS}
0<k_{n+4}+c_1k_{n+3}+c_2 k_{n+2}+c_1k_{n+1}+k_n\le 1
\end{equation}
where $c_1=\alpha_1+\alpha_2$ and $c_2=\alpha_1\alpha_2+2$. 
Here $\alpha_i\in (-2,2)$ are arbitrary chosen constants.
The bijective map $T$ on $\Z^4$:
$$
(k_{n},k_{n+1},k_{n+2},k_{n+3}) \mapsto (k_{n+1},k_{n+2},k_{n+3}, -\lceil c_1 k_{n+3}+c_2 k_{n+2}+c_1 k_{n+1}+k_n -1\rceil )
$$
is conjugate to (\ref{ShiftedAddition}).
By bijectivity, the orbit is purely periodic if and only if it is eventually periodic. Moreover, the periodicity is equivalent to the boundedness of the orbit.
% Since $T((0,0,0,0))= (0,0,0,1)$, the initial point $z_0$ for (\ref{ShiftedAddition}) corresponds to the image of the origin and
We are interested in the recurrence of the orbit 
of $(k_{-4},k_{-3},k_{-2},k_{-1})
=(0,0,0,0)$ by $T$.
This map approximates a linear map $\Phi$ defined by
$$
 \begin{pmatrix} x_1\cr x_2\cr x_3 \cr x_4\end{pmatrix}
 \mapsto  \begin{pmatrix} 0& 1 & 0 & 0 \cr
                                             0& 0 & 1 & 0\cr
                                             0 & 0& 0 & 1\cr
                                          -1  & -c_1 & -c_2 & -c_1\end{pmatrix}  
\begin{pmatrix} x_1\cr x_2\cr x_3 \cr x_4\end{pmatrix}
$$
for $(x_1,x_2,x_3,x_4)^T\in \R^4$. 
The map $\Phi$ has four eigenvalues $\exp(\pm \theta_i \sqrt{-1})$
with $\alpha_i=-2\cos(\theta_i)$  for $i=1,2$.
Therefore the map $T:\Z^4\rightarrow \Z^4$ is understood as a discretized version of rotation. 
A simpler case:
the discretized rotation in $\Z^2$ is extensively studied in the literature. It is defined similarly 
by a recurrence 
\begin{equation}
\label{SRS}
0\le a_{n+2}+\lambda a_{n+1}+a_n<1,\quad a_n\in \Z
\end{equation}
with a fixed $\lambda\in (-2,2)$.
A notorious conjecture states that any sequence produced by this recursion is
periodic for any initial vector $(a_0,a_1)\in \Z^2$. 
The validity is known only for $11$ values of $\lambda$, 
see \cite{LHV,Akiyama-Brunotte-Pethoe-Steiner:07,KLV}.

Note that
%that we have no reason to believe that orbits in (\ref{4SRS}) are bounded when %
if $\alpha_1=\alpha_2$ then there is an unbounded real sequence $
(b_n)_{n\in \N}$
which satisfies
$$
b_{n+4}+c_1b_{n+3}+c_2 b_{n+2}+c_1b_{n+1}+b_n=0
$$
because of the shape of general terms of this recurrence.
%is $\Re((a+bn)\exp(n\theta_1 \sqrt{-1}))$
% with some $a,b\in \C$.
In particular, if $\alpha_1=\alpha_2\in \{-1,0,1\}$ then $\{T^n(0,0,0,0)|\ (n\in\N)\}$ is unbounded.
%it with $(b_n,b_{n+1},b_{n+2},b_{n+3})\in \Z^4$.
%Therefore at least one of  $(k_n)_{n \in \N}$ and
%$(k_n+b_n)_{n \in \N}$ is unbounded
%and satisfies (\ref{4SRS}). 
Thus we can not expect\footnote{On the other hand, there are points with
$\alpha_1=\alpha_2$
where the $T$-orbits of $(0,0,0,0)$
is periodic, e.g., $(\alpha,\alpha)$ with $\alpha\in (1-\sqrt{3},-2/3]\cup [-1/4,-2/9]$.} the boundedness of the orbits of $T$ 
when $\alpha_1=\alpha_2$.
Excluding these cases, a natural generalization of the above conjecture for (\ref{SRS}) would be a

\begin{conj}
\label{4Conj}
If $c_2>2|c_1|-2$ and $c_2-2<\frac{c_1^2}4<4$, then for any initial vector
$(k_1,k_2,k_3,k_4)\in \Z^4$ the sequence satisfying (\ref{4SRS}) is periodic,
\end{conj}

\noindent
since $c_2>2|c_1|-2$ and $c_2-2<\frac{c_1^2}4<4$ is equivalent to $\alpha_i\in (-2,2)$ for $i=1,2$ and $\alpha_1\neq \alpha_2$.

We are pessimistic about its validity, due to the existence of 
very large orbits. 
However, even if Conjecture \ref{4Conj} may not hold, it could be true 
for almost all cases. 

Let us restrict ourselves to the orbit of the origin.
Since periodic orbits are often dominant in zero entropy systems, unbounded orbits may not give a contribution of 
positive measure in \S 4 and period cells would exhaust the total square $(-2,2)^2$ in measure. We propose a weaker

\begin{conj}
\label{4ConjMeas}
Letting $(k_{-4},k_{-3},k_{-2},k_{-1})=(0,0,0,0)$, the sequence satisfying (\ref{4SRS}) is periodic for all most all $(\alpha_1,\alpha_2)\in (-2,2)^2$ in measure.
\end{conj}

Here our measure is equivalent to the two-dimensional Lebesgue measure.
See Figure \ref{SalemParry3} for period cells. Black dots are the points where the orbit of the origin might be unbounded.

Note that once we find a period of $T$ starting from the origin, 
Lemma \ref{Str1} and Lemma \ref{Str2} give us a (possibly degenerated)
period cell.
This fact 
is clear when $g_i$ does not visit $1$ and we can apply Lemma \ref{Str1} with
$$
u=\frac 12 \min_{i} g_i, \quad v=\frac 12(1+\max_i g_i).
$$
For Lemma \ref{Str2}, it looks like we have additional 
constraints. We shall show that this is not the case. 
Taking the period $p\in \N$ 
with $T^p((0,0,0,0))=(0,0,0,0)=(k_{-4},k_{-3},k_{-2},k_{-1})$, we have
$$
g_n=t_{n-1}(1)+k_n=k_{n-4}+c_1 k_{n-3} + c_2 k_{n-2} + c_1k_{n-1}+ k_{n}
$$
for $n=0,\dots,p$ and
$$
k_0=1,\ k_1=-\lceil c_1-1 \rceil.
$$
If $g_i=1$ occurs as a polynomial of $\Z[\alpha_1,\alpha_2]$ with $1\le i \le p-1$, then $k_{i-2}=k_{i-3}+k_{i-1}=0$ and $k_{i-4}+k_i=1$.
From $g_{i\pm 1}\in (0,1]$, we have
%$$
%0<k_{i-5}+c_1 k_{i-4} + c_2 k_{i-3} + c_1 k_{i-2}+ k_{i-1}\le 1
%$$
%and
%$$
%0<k_{i-3}+c_1 k_{i-2} + c_2 k_{i-1} + c_1k_{i}+ k_{i+1}\le 1.
%$$
$$
0<g_{i-1}=k_{i-5}+c_1 k_{i-4} + (c_2-1) k_{i-3}\le 1
$$
and
$$
0<g_{i+1}=(1-c_2)k_{i-3} + c_1(1-k_{i-4})+ k_{i+1}\le 1.
$$
These imply
$$
k_{i-5}=-\lceil c_1 k_{i-4} + (c_2-1) k_{i-3}-1 \rceil
$$
and
$$
k_{i+1}=-\lceil (1-c_2)k_{i-3} + c_1(1-k_{i-4}) -1 \rceil.
$$
We see
$$
k_1\le k_{i-5}+k_{i+1}
$$
from an easy fact
$$
\lceil x \rceil + \lceil y \rceil - \lceil x+y+1 \rceil \in \{-1,0\}
$$
for any $x,y\in \R$. Therefore when $g_i=1$, we have
$$
g_{i-1}+g_{i+1}=k_{i-5} + c_1+ k_{i+1} \ge c_1+k_1=g_1.
$$
Since $g_1=g_{i+1}+g_{i-1}$ implies $g_1>g_{i+1}$, 
we can apply Lemma \ref{Str2} in any case with
$$
u=\frac 12 \min\left(\min_i g_i, \min_{\substack{g_i=1\\
g_1=g_{i-1}+g_{i+1}}} (g_1-g_{i+1})\right),
\quad v=\frac 12(1+\max_{g_i\neq 1} g_i).
$$
Therefore we can apply Lemma \ref{Str2} for every period starting from the origin.
Since we expect such period cells to cover $(-2,2)^2$ in measure as in 
Conjecture \ref{4ConjMeas}, for any 
$\eps>0$, we will find a finite union of period cells whose measure is not less than $1-\eps$.
Following the same proof as Theorem \ref{main2}, we arrive at a plausible
\begin{conj}
\label{c6}
$c(6)=1$.
\end{conj}
For general Salem number $\beta$'s, we numerically observe many
$(m,p)$-periodic $d_{\beta}(1)$'s with $m>1$. 
Interestingly, we find no role of $(m,p)$-periods with $m>1$
in the above discussion.
Every period of $T$ gives rise to $(1,p)$-periodic $d_{\beta^m}(1)$ 
and other orbits of $T$ are aperiodic if they exist.
Of course
this does not cause any contradiction, since we are studying sufficiently large $\beta$ with respect to the location of the conjugates.
%\newpage
\bigskip

\begin{center}
{\bf Appendix 1.}
\end{center}

Additional polynomials to improve $c(6)$:
\begin{align*}
R_{19}=&1+3 x+5 x^2+5 x^3+3 x^4+x^5,\\
R_{20}=&1+3 x+5 x^2+6 x^3+5 x^4+3 x^5+x^6,\\
R_{21}=&1+3 x+5 x^2+6 x^3+6 x^4+5 x^5+3 x^6+x^7, \\
R_{22}=&1+3 x+5x^2+6 x^3+6 x^4+6 x^5+6 x^6+5 x^7+3 x^8+x^9,\\
R_{23}=&1+3 x+4x^2+4 x^3+4 x^4+4 x^5+3 x^6+x^7,\\
R_{24}=&1+3 x+4 x^2+3 x^3+2 x^4+3 x^5+4 x^6+3 x^7+x^8,\\
R_{25}=&1+x+x^2+2 x^3+2 x^4+2 x^5+2x^6+x^7+x^8+x^9,\\
R_{26}=&1+x+x^2+2 x^3+2 x^4+2 x^5+3 x^6+2 x^7+2 x^8+2 x^9+x^{10}+x^{11}+x^{12},\\
R_{27}=&1+4 x+8 x^2+11 x^3+11 x^4+8 x^5+4 x^6+x^7,\\
R_{28}=&1+x+2 x^2+2 x^3+2 x^4+3 x^5+2 x^6+2 x^7+2 x^8+x^9+x^{10},\\
R_{29}=&1+x+2 x^2+2 x^3+3 x^4+3 x^5+3 x^6+3 x^7+2 x^8+2 x^9+x^{10}+x^{11},\\
R_{30}=&1+x+2 x^2+3 x^3+3 x^4+4 x^5+4 x^6+4 x^7+4 x^8+3 x^9+3 x^{10}+2 x^{11}+x^{12}+x^{13},\\
R_{31}=&1+2x^2-x^3+2 x^4-2 x^5+2 x^6-2 x^7+2 x^8-x^9+2 x^{10}+x^{12},\\
R_{32}=&1+2 x+2 x^2+2 x^3+3 x^4+3 x^5+2 x^6+2 x^7+2 x^8+x^9,\\
R_{33}=&1+2x+2 x^2+3 x^3+4 x^4+4 x^5+4 x^6+4 x^7+3 x^8+2 x^9+2 x^{10}+x^{11},\\
L_6=&1-3 x+5x^2-5 x^3+3 x^4-2 x^6+2 x^7-2 x^9+3 x^{10}-2 x^{11}+2 x^{13}-2 x^{14}\\
&+3 x^{16}-5 x^{17}+5 x^{18}-3x^{19}+x^{20},\\
L_7=&1-3 x+6 x^2-8 x^3+8 x^4-5 x^5+5 x^7-7 x^8+5 x^9-5 x^{11}+8 x^{12}-8 x^{13}\\
&+6 x^{14}-3 x^{15}+x^{16}
\end{align*}

\begin{figure}[htb]
\includegraphics[width=\columnwidth]{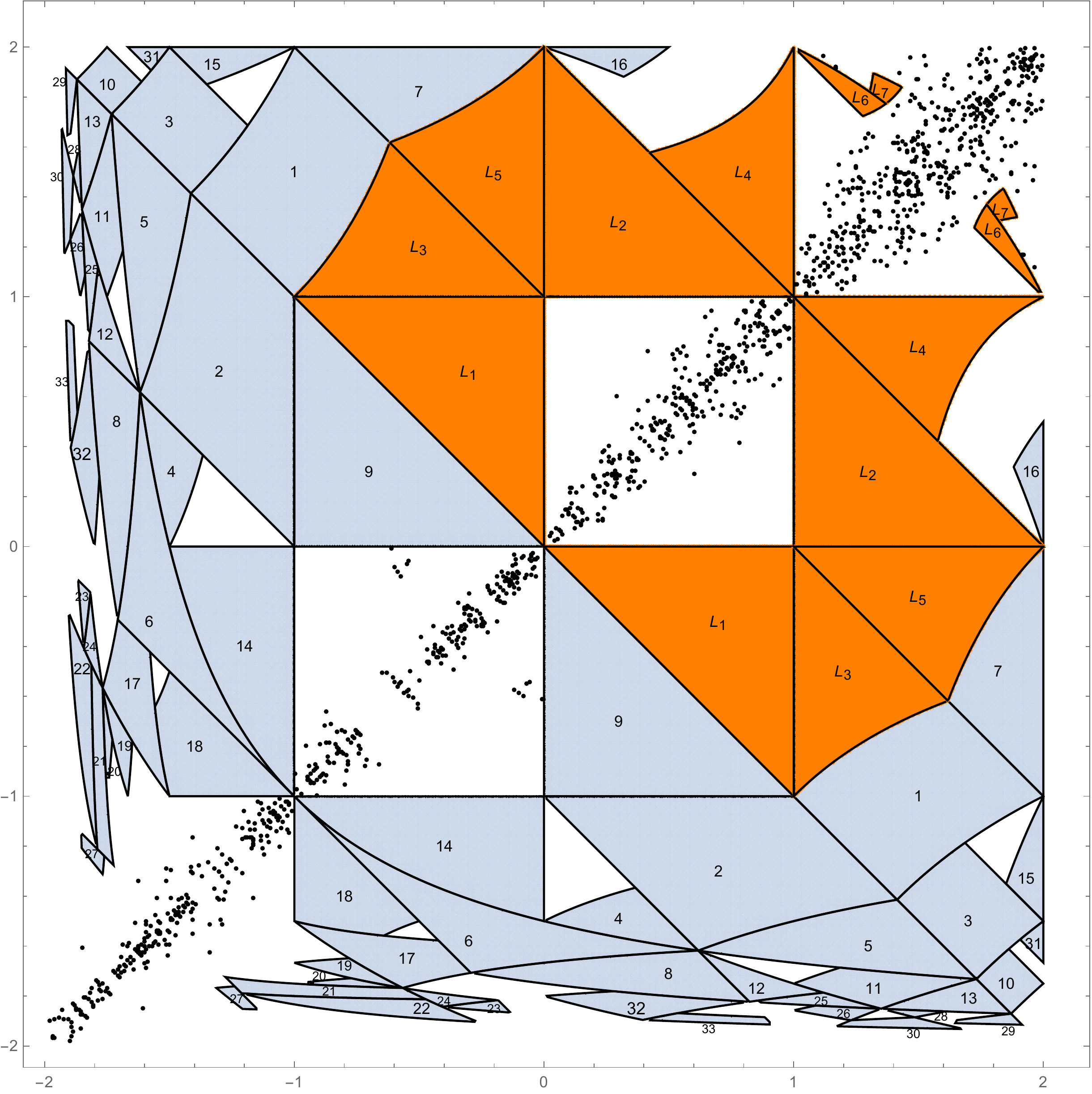}
\caption{Period cells occupy more than half in measure\label{SalemParry3}. The orbit of the origin does not form a period 
until 50,000 iterations at black dots.}
\end{figure}

One can check that 
these 40 regions are mutually disjoint by symbolic computation.
This gives the estimate $c(6)\ge 0.505254$, see Figure \ref{SalemParry3}. 

In the list of the regions in \S \ref{Proof2}, all 
the inequalities defining the period cells in $y<x$
are of the form $A < y$ or $y \le B$ with some $A$ and $B$. 
This is no longer true for the polynomial $L_6$ and $L_7$. 
The cell for $L_6$ is an open set, and the common boundary of the two cells belongs to the one for $L_7$. Thus
an inequality of the form $A\le y$ is required for the cell for $L_7$ 
in $y<x$. 
The point $(\alpha_1,\alpha_2)=((11+\sqrt{2})/7,(11-\sqrt{2})/2)$ is the end point of the common boundary but belongs to neither of them. 
Applying our algorithm at this point, 
we obtain a polynomial of degree $412$ and the corresponding cell in $y<x$
degenerates to a singleton $(\alpha_1,\alpha_2)$.

\bigskip

\begin{center}
{\bf Appendix 2.} 
\end{center}
\medskip

Our strategy in this paper is to study sufficiently large $\beta$. However, 
for regions $\RR_9$ and $\LL_1$, we can remove this adjective
``sufficiently large". Indeed, at the beginning of this study, 
we found Proposition \ref{epnsion of reggular} below and then generalized it
to our current setting. 
In particular, the formulation of
Lemma \ref{Str2} is inspired by the case (2) of Proposition \ref{epnsion of reggular}.
Let $P$ be the minimum polynomial of a sextic Salem number:

\begin{eqnarray}
\label{Salem of degree6}
P(x)=x^6-ax^{5}-bx^{4}-cx^{3}-bx^{2}-ax+1.
\end{eqnarray}

We denote by $Q$ its trace polynomial:
\begin{eqnarray}
\label{trace6}
Q(y)=y^3-ay^2-(b+3)y-(c-2a)=(y-\gamma)(y-\alpha_1)(y-\alpha_2)
\end{eqnarray}

We say that $\beta$ is {\bf well-posed} if its trace polynomial $Q(y)$ has three roots $\gamma,\alpha_1,\alpha_2$ such that
$$
-1<\alpha_1<0<\alpha_2<1<2<\gamma.
$$
Then we have a

\begin{lem}\label{Salem 6 of t3}
A real number $\beta >1$ is a well-posed Salem number of degree 6 
if and only if $\beta$  is the dominant root of the polynomial in $\Z[x]$ of the form given by (\ref{Salem of degree6}) satisfying the following conditions:
\begin{enumerate}
\item  $2-2b<2a+c$.
\item $c<2a$.
\item  $|b+2|<c-a$.
\end{enumerate}
\end{lem}

\begin{proof}  
Since $Q$ is cubic, the well-posedness is equivalent to $Q(-1)<0,Q(0)>0,Q(1)<0,Q(2)<0$. We see (1) $\Leftrightarrow Q(2)<0$, (2) $\Leftrightarrow Q(0)>0$, and 
$Q(\pm1)<0 \Leftrightarrow$  (3). 
\end{proof}

Classifying into $b<-1$ and $b\ge -1$, we obtain the following Proposition, which gives a partial response to Problem 1 in \cite{J.-L.Verger-Gaugry2008}.

\begin{prop}\label{epnsion of reggular} 
Let $\beta $ be a well-posed Salem number of minimum polynomial $P$ in (\ref{Salem of degree6}). Then $\beta$ is Parry number and we have: 
\begin{enumerate}
\item If $2\leq -b<c-a+2$,  then
\begin{align}
\label{Case1}
d_{\beta}(1)=&a-1(a+b+1,c-a+b+1,c-a-1,2a-c-1,\\
\nonumber
&2a-c-1,c-a-1,c-a+b+1,a+b+1,a-2,a-2)^{\infty}
\end{align}
\item If $a-c+2<-b\leq 1$ then 
\begin{align}
\label{Case2}
d_{\beta}(1)=&a(b+1,c-a-1,a-1,2a+b-c+1,c-a-1,\\
\nonumber
&c-a-1,2a+b-c+1,a-1,c-a-1,b+1,a-1,a-1)^{\infty}
\end{align}
\end{enumerate}
\end{prop}

\begin{proof}
Lemma \ref{Salem 6 of t3} (3) implies $c-a\ge 1$ and (2) gives $a\ge 2$
and $c\ge 3$.
Starting from the representation of zero in base $\beta$:
$$
-1,a,b,c,b,a,-1
$$
we can construct another representation as
\bigskip

\begin{scriptsize}
\begin{tabular}{c|c|c|c|c|c|c|c|c|c|c|c|}
-1& a & b &  c& b & a & -1&   &   &   &   & \\
  & -1& a & b & c & b & a & -1&   &   &   & \\
  &   & 1 & -a&-b &-c & -b& -a&  1&   &   & \\
  &   &   & 1 & -a&-b &-c & -b& -a&  1&   & \\
  &   &   &   &-1 & a & b &  c&  b&  a& -1&   \\
  &   &   &   &   &-1 & a & b &  c&  b& a & -1\\
\hline
-1& a-1& a+b+1& c-a+b+1 & c-a-1& 2a-c-1 & 2a-c-1& c-a-1& c-a+b+1& a+b+1 & a-1&-1\\
\end{tabular}
\end{scriptsize}
\bigskip

Performing recursive shifted addition of this new 
representation in base $\beta$,
we obtain the infinite representation (\ref{Case1}). 
We can check the condition (\ref{Selflex}) from Lemma 
\ref{Salem 6 of t3}. Similarly, we have
\bigskip

\begin{scriptsize}
\begin{tabular}{c|c|c|c|c|c|c|c|c|c|c|c|c|c|}
-1& a & b &  c& b & a & -1&   &   &   &   &   &   & \\
  &   & 1 & -a& -b& -c& -b& -a&  1&   &   &   &   & \\
  &   &   & -1& a & b &  c& b & a & -1&   &   &   & \\
  &   &   &   &-1 & a & b &  c&  b&  a& -1&   &   & \\
  &   &   &   &   & 1 & -a& -b& -c& -b& -a&  1&   & \\
  &   &   &   &   &   &   & -1 & a&  b&  c&  b&  a&-1\\
\hline
-1& a& b+1&c-a-1&a-1& 2a+b-c+1& c-a-1 & c-a-1& 2a+b-c+1& a-1&c-a-1& b+1& a &-1 \\
\end{tabular}
\end{scriptsize}
\bigskip

By recursive shifted addition, we obtain the representation 
(\ref{Case2}). The condition (\ref{Selflex}) follows
from Lemma \ref{Salem 6 of t3}.
\end{proof}

%\begin{figure}[htb]
%\includegraphics[width=\columnwidth]{Large.pdf}
%\caption{The orbit of the origin does not return by $50,000$ iterations at the dots $(\alpha_1,\alpha_2)$ \label{Large}}
%\end{figure}

\section*{Acknowledgments}

We would like to thank the anonymous referee for the careful reading of the manuscript. This research was partially supported by JSPS grants (20K03528, 17K05159, 21H00989).

%\bibliography{allbiblio}
%\bibliographystyle{amsplain}

\end{document}